\documentclass[a4paper,titlepage,11pt]{article}
\usepackage{hyperref}
\usepackage[usenames,dvipsnames]{pstricks} 
\usepackage{pst-plot} 
\usepackage{amssymb,amsthm,amsmath} 
\usepackage[a4paper]{geometry}
\usepackage{datetime2}
\usepackage[utf8]{inputenc}
\usepackage[italian,english]{babel}
\usepackage{mathrsfs}
\usepackage{bbm}

\newcommand{\R}{\mathbb{R}}

\newcommand{\s}{\mathbb{S}}
\newcommand{\M}{\mathcal{M}}

\newcommand{\N}{\mathbb{N}}
\newcommand{\ep}{\varepsilon}

\newcommand{\Lip}{\operatorname{Lip}}
\newcommand{\Leb}{\ensuremath{\mathscr{L}}}

\newtheorem{thm}{Theorem}[section]
\newtheorem{thmbibl}{Theorem}

\newtheorem{rmk}[thm]{Remark}

\newtheorem{defn}[thm]{Definition}

\newtheorem{lemma}[thm]{Lemma}

\newtheorem{question}[thmbibl]{Question}

\title{New estimates for a class of non-local approximations of the total variation}

\author{
Nicola Picenni\vspace{1ex}\\ 
{\normalsize Scuola Normale Superiore} \\
{\normalsize PISA (Italy)}\\
{\normalsize e-mail: \texttt{nicola.picenni@sns.it}}
\\[2ex]
{\normalsize Università di Pisa} \\
{\normalsize Dipartimento di Matematica}\\ 
{\normalsize PISA (Italy)}\\
{\normalsize e-mail: \texttt{nicola.picenni@dm.unipi.it}}
}
\date{}

\begin{document}

\maketitle

\begin{abstract}

We consider a class of non-local functionals recently introduced by H.~Brezis, A.~Seeger, J.~Van~Schaftingen, and P.-L.~Yung, which offers a novel way to characterize functions with bounded variation. 

We give a positive answer to an open question related to these functionals in the case of functions with bounded variation. Specifically, we prove that in this case the liminf of these functionals can be estimated from below by a linear combination in which the three terms that sum up to the total variation (namely the total variation of the absolutely continuous part, of the jump part and of the Cantor part) appear with different coefficients. We prove also that this estimate is optimal in the case where the Cantor part vanishes, and we compute the precise value of the limit in this specific scenario.

In the proof we start by showing the results in dimension one by relying on some measure theoretic arguments in order to identify sufficiently many disjoint rectangles in which the difference quotient can be estimated, and then we extend them to higher dimension by a classical sectioning argument.

%
%
%

\vspace{6ex}

\noindent{\bf Mathematics Subject Classification 2020 (MSC2020):} 26B30, 49J45, 49Q20.

\vspace{6ex}


\noindent{\bf Key words:} 
Functions of bounded variation, special functions of bounded variation, non-local functionals.

\end{abstract}

\section{Introduction}

In this paper we consider a class of non-local, non-convex functionals related to the total variation. In order to introduce these functionals, let $N\geq 1$ be a positive integer, let $\Omega \subset \R^N$ be an open set and let $\gamma\in \R$ and $\lambda\in (0,+\infty)$ be two real parameters.

Let $\nu_\gamma$ be the measure on $\R^N\times \R^N$ defined by
$$\nu_\gamma (A):=\iint_A |y-x|^{\gamma-N}\,dx\,dy,$$
for every measurable subset $A\subseteq \R^{N}\times \R^{N}$.

For every measurable function $u:\Omega\to\R$ let us set
\begin{equation*}
E_{\gamma,\lambda}(u,\Omega):=\left\{(x,y)\in \Omega\times\Omega:|u(y)-u(x)|>\lambda |y-x|^{1+\gamma}\right\}.
\end{equation*}

Now we can set
\begin{equation}\label{defn:F}
F_{\gamma,\lambda}(u,\Omega):=\lambda \nu_\gamma(E_{\gamma,\lambda}(u,\Omega)).
\end{equation}

In the case $\Omega=\R^N$ we simply write $E_{\lambda,\gamma}(u)$ and $F_{\gamma,\lambda}(u)$ instead of $E_{\lambda,\gamma}(u,\R^N)$ and $F_{\gamma,\lambda}(u,\R^N)$.

The functionals (\ref{defn:F}) were introduced in \cite{2022-BSVY-Lincei,BSVY} and generalize some other families of functionals previously considered in the literature. In particular, in the case $\gamma=-1$, the family $\{F_{-1,\lambda}\}$ was first studied in \cite{2006-Nguyen-JFA,2011-Nguyen-Duke} (see also the more recent developments in \cite{2018-CRAS-AGMP,2020-APDE-AGMP,2020-APDE-AGP,2018-AnnPDE-BN}). The main results in this case are that
$$\lim_{\lambda\to 0^+} F_{-1,\lambda}(u)=C_N \int_{\R^N} |\nabla u(x)|\,dx=C_N \|Du\|_\M \qquad \forall u \in C^1_c(\R^N),$$
where
\begin{equation}\label{defn:C_N}
C_N:=\int_{\s^{N-1}} |x_1|\,d\mathcal{H}^{N-1}(x),
\end{equation}
and that, quite surprisingly,
$$\Gamma-\lim_{\lambda\to 0^+} F_{-1,\lambda}(u)=C_N \log2 \cdot \|Du\|_\M \qquad \forall u \in L^1(\R^N),$$
where $\|Du\|_\M$ denotes the total variation of $u$, which is intended to be equal to $+\infty$ if $u\in L^1(\R^N)\setminus BV(\R^N)$.

On the other hand, in the case $\gamma=N$ we recover the quantities considered in \cite{2021-BVY-PNAS,2021-BVY-CalcVar,2022-Poliakovsky-JFA}, since in this case $\nu_N$ is equal to the Lebesgue measure $\Leb^N$ and
$$\left[\frac{|u(y)-u(x)|}{|y-x|^{1+N}}\right]_{L^{1,\infty}(\Omega)}=\sup_{\lambda>0}F_{N,\lambda}(u,\Omega),$$
where $[\cdot]_{L^{1,\infty}(\Omega)}$ denotes the weak $L^1$ quasi-norm.

In this paper, motivated by \cite[Section~9]{Brezis-Lincei}, we limit ourselves to the case $\gamma>0$. In this case, it is known (see \cite[Theorem~1.4]{BSVY}) that there exist two constants $c_1(N,\gamma)$ and $c_2(N,\gamma)$ such that
\begin{equation}\label{est:sup_F}
c_1(N,\gamma)\|Du\|_\M\leq \sup_{\lambda>0} F_{\gamma,\lambda}(u)\leq c_2(N,\gamma)\|Du\|_\M,
\end{equation}
for every $u\in L^1 _{loc}(\R^N)$. In particular, it follows that
$$\sup_{\lambda>0} F_{\gamma,\lambda}(u)<+\infty$$
if and only if $u$ belongs to the space $\dot{BV}(\R^N)$ of the functions in $L^1_{loc}(\R^N)$ with globally bounded variation.

Moreover, for every $\gamma>0$ it holds that (see \cite[Theorem~1.1]{BSVY})
\begin{equation}\label{lim_W11}
\lim_{\lambda\to +\infty} F_{\gamma,\lambda}(u)=\frac{C_N}{\gamma} \int_{\R^N} |\nabla u(x)|\,dx,
\end{equation}
for every $u$ in the space $\dot{W}^{1,1}(\R^N)$ of functions in $L^1_{loc}(\R^N)$ with $\nabla u\in L^1(\R^N)$, where $C_N$ is the constant defined in (\ref{defn:C_N}).

The extension of (\ref{lim_W11}) to the case in which $u\in \dot{BV}(\R^N)$ is not straightforward. Indeed, it was proved in \cite[Lemma~3.6]{BSVY} that if $u$ is the characteristic function of a bounded convex domain with smooth boundary, then for every $\gamma>0$ it holds that
\begin{equation}\label{lim_salti}
\lim_{\lambda\to +\infty} F_{\gamma,\lambda}(u)=\frac{C_N}{\gamma+1}\|Du\|_\M<\frac{C_N}{\gamma}\|Du\|_\M.
\end{equation}

This result suggests that the singular part of the derivative contributes to the limit of $F_{\gamma,\lambda}(u)$ in a different way with respect to the absolutely continuous part, but anyway it seems that it does add a positive contribution in the limit.

This led to the following questions, which were raised in \cite[Section~9]{Brezis-Lincei} and in \cite[Section~7.2]{BSVY}.
\begin{question}\label{question_constant}
Let $\gamma>0$, and let $u:\R^N\to\R$ be a measurable function such that
$$\liminf_{\lambda\to+\infty} F_{\gamma,\lambda}(u)=0.$$
Can we conclude that $u$ is constant (almost everywhere)?
\end{question}

\begin{question}\label{question_TV}
Let $\gamma>0$. Is there a positive constant $c(N,\gamma)>0$ such that
$$\liminf_{\lambda\to+\infty} F_{\gamma,\lambda}(u)\geq c(N,\gamma) \|Du\|_\M,$$
for every measurable function $u:\R^N\to\R$, with the usual understanding that $\|Du\|_\M=+\infty$ if $u\notin \dot{BV}(\R^N)$?
\end{question}

The main contribution of the present paper is a positive answer to these questions in the case $u\in \dot{BV}(\R^N)$. 

Before stating the result, we recall that for a function $u\in\dot{BV}(\R^N)$ we can decompose its distributional derivative $Du$ as the sum of three finite $\R^N$-valued measures, which are supported on disjoint sets: the absolutely continuous part $D^a u$, the jump part $D^j u$, and the Cantor part $D^c u$ (see \cite[Section~3.9]{AFP}).

Our first main result is the following.
\begin{thm}\label{thm:liminf}
For every $\gamma>0$ and every $u\in \dot{BV}(\R^N)$ it turns out that
\begin{equation}\label{th:liminf}
\liminf_{\lambda\to+\infty} F_{\gamma,\lambda}(u)\geq \frac{C_N}{\gamma} \|D^{a}u\|_\M+\frac{C_N}{1+\gamma}\|D^{j}u\|_\M+ \frac{C_N(2^{\gamma}-1)}{\gamma(1+\gamma)2^{2+\gamma}}\|D^{c}u\|_\M.
\end{equation}
\end{thm}

Our second main result shows that the constants appearing in front of the absolutely continuous part and the jump part are optimal, and unifies (\ref{lim_W11}) and (\ref{lim_salti}) in the case in which $u\in \dot{SBV}(\R^N)$, namely $u\in \dot{BV}(\R^N)$ and $D^c u=0$.
\begin{thm}\label{thm:lim_SBV}
For every $\gamma>0$ and every $u\in \dot{SBV}(\R^N)$ it turns out that
$$\lim_{\lambda\to+\infty} F_{\gamma,\lambda}(u) = \frac{C_N}{\gamma} \|D^{a}u\|_\M+\frac{C_N}{1+\gamma}\|D^{j}u\|_\M.$$
\end{thm}

\begin{rmk}
\begin{em}
For the sake of simplicity, in this paper we consider only the case in which $\Omega=\R^N$. However, it should not be difficult to extend the main results to the case of bounded regular domains. In particular, in the case of convex domains, this should be almost straightforward, since all their one-dimensional sections are just open intervals.
\end{em}
\end{rmk}

\paragraph{\textmd{\textit{Overview of the technique}}}

Let us summarize the main ideas in the proofs of our results, which are different from those used in \cite{BSVY,2022-Poliakovsky-JFA}, since we do not exploit the BBM formula \cite{BBM,Davila-BBM}.

For both theorems, we first establish the result in the one-dimensional case, and then we extend it to the case of higher space dimensions by a sectioning argument. This is a standard but quite effective tool in this kind of problems (see \cite{2020-APDE-AGMP,2022-GP-CCM}).

The proof of the one-dimensional version of Theorem~\ref{thm:liminf} relies on some measure-theoretic arguments that allow us to find sufficiently many disjoint rectangles inside $E_{\gamma,\lambda}$ on which we can control the difference quotient of $u$.

As for Theorem~\ref{thm:lim_SBV}, in view of Theorem~\ref{thm:liminf}, it is enough to prove an estimate from above for the limsup of $F_{\gamma,\lambda}(u)$. To this end we exploit an argument from \cite[Section~3.4]{BSVY} which basically shows that it is enough to prove such an estimate for a class of functions $u$ that is dense in $\dot{SBV}(\R^N)$ with respect to the \emph{strong} BV topology. In dimension one such a class is provided by functions with finitely many jump points that are smooth and Lipschitz continuous in every interval that does not contain such points. For functions in this class, the limit can be easily computed.

\paragraph{\textmd{\textit{The case $p>1$}}}
The functionals considered in the paper \cite{BSVY} actually depend also on a parameter $p\geq 1$, and can be written as
$$F_{p,\gamma,\lambda}(u,\Omega):=\lambda^p \nu_{\gamma}(E_{{\gamma/p},\lambda}(u,\Omega)).$$

The same is true for the various special cases previously considered in
\cite{2006-Nguyen-JFA,2011-Nguyen-Duke,2020-APDE-AGMP,2021-BVY-PNAS,2021-BVY-CalcVar,2022-Poliakovsky-JFA}.

This higher generality allows to obtain characterizations of the Sobolev spaces $W^{1,p}(\R^N)$ or $\dot{W}^{1,p}(\R^N)$, and also for more general types of spaces (see \cite{DLYYZ_JFA,DLYYZ_CalcVar,ZYY_CCM,ZYY_CalcVar}), together with estimates on their semi-norms.

In this paper we only consider the case $p=1$, which is the most challenging, because the gap between $\dot{W}^{1,1}(\R^N)$ and $\dot{BV}(\R^N)$ creates additional difficulties, and we refer to \cite{BSVY} and to the references therein for the numerous interesting results in the case $p>1$. 


\paragraph{\textmd{\textit{Recent developments}}}
After this work was completed, Lahti in \cite{2023-Lahti} proved with different techniques a sharper version of the estimate (\ref{th:liminf}), namely
\begin{equation}\label{est:Lahti}
\liminf_{\lambda\to+\infty} F_{\gamma,\lambda}(u)\geq \frac{C_N}{\gamma} \|D^{a}u\|_\M+\frac{C_N}{1+\gamma}\|D^{j}u\|_\M+ \frac{C_N}{1+\gamma}\|D^{c}u\|_\M,
\end{equation}
for every $u \in \dot{BV}(\R^N)$ and $\gamma>0$.

He also found a Cantor-type function $u\in BV(0,1)$ for which $Du=D^c u$ and
$$\liminf_{\lambda\to+\infty} F_{\gamma,\lambda}(u,(0,1))= \frac{C_1}{1+\gamma}\|D^{c}u\|_\M,$$
thus establishing the optimality of the constant in front of the Cantor part in (\ref{est:Lahti}).

On the other hand, obtaining good estimates from above for functions with non-vanishing Cantor part (in order to extend Theorem~\ref{thm:lim_SBV}) is quite complicated, since there is not a nice class of functions which is strongly dense in $\dot{BV}(\R)$. Indeed, it is not possible to approximate functions with non-vanishing Cantor part in the strong BV topology with functions without a Cantor part. It is also conceivable that for some functions the liminf might differ from the limsup or that different functions might produce different constants in the limit, since these phenomena occur in some similar contexts (see the pathologies in \cite{2018-AnnPDE-BN}).

We also mention another very recent development obtained in \cite{2023-GP-Gamma-liminf}, where it is proved that
$$\Gamma-\liminf_{\lambda\to +\infty} F_{\gamma,\lambda}(u) \geq C_N\cdot \frac{\log 2}{2^{\gamma+1}-1}\cdot \|u\|_{\M},$$
for every $u \in L^1 _{loc}(\R^N)$ and $\gamma>0$, thus providing a positive answer to Question~\ref{question_constant} and Question~\ref{question_TV} also for $u\in L^1_{loc}(\R^N)\setminus \dot{BV}(\R^N)$, in which case the right-hand side is infinite.

\paragraph{\textmd{\textit{Structure of the paper}}}

The paper is organized as follows. In Section~\ref{sec:1d}, after recalling some basic properties of functions of bounded variation in one dimension, we prove the main results in the one-dimensional case. Then, in Section~\ref{sec:Nd} we show how the problem can be reduced to the one-dimensional setting by a sectioning argument, and we complete the proofs in the higher dimensional case.

\section{The one-dimensional case}\label{sec:1d}

In this section we prove our main results in the one-dimensional case $N=1$. We observe that in this case in (\ref{defn:C_N}) we have $C_1=2$ and that we can rewrite the functional in the following more convenient way
\begin{equation}\label{F_E'}
F_{\gamma,\lambda}(u,\Omega)=C_1\lambda \nu_\gamma(E_{\gamma,\lambda}'(u,\Omega)),
\end{equation}
where
$$E_{\gamma,\lambda}'(u,\Omega):=\left\{(x,y)\in \Omega\times\Omega: x<y\mbox{ and } |u(y)-u(x)|>\lambda |y-x|^{1+\gamma}\right\}.$$

We also recall that functions of bounded variation in one dimension have some special properties, that we list in the following lemma.

\begin{lemma}\label{lemma:BV_1d}
Let $u \in \dot{BV}(\R)$. Then we can choose a representative, that we still denote with $u$, satisfying the following properties.
\begin{itemize}
\item $u$ is differentiable in the classical sense at almost every $x\in \R$ and $D^a u=u'\cdot \Leb^1$, where $\Leb^1$ is the one-dimensional Lebesgue measure.
\item $u$ admits a left limit $u(x_-)$ and a right limit $u(x_+)$ at every $x\in \R$, and they coincide for every $x$ outside a set $J_u$ that is at most countable. With these notations, it holds that
$$D^j u=\sum_{x\in J_u} (u(x_+)-u(x_-))\delta_{x}.$$
\item The Cantor part of the derivative is supported on the set $C=C_+\cup C_-$, where
$$C_\pm:=\left\{x\in \R\setminus J_u : \lim_{h\to 0} \frac{u(x+h)-u(x)}{h} =\pm\infty \mbox{\emph{ and }}\lim_{h\to 0^+} \frac{|D^c u ([x,x+h])|}{|D^c u| ([x,x+h])}=1 \right\}.$$
\end{itemize}
\end{lemma}

\begin{proof}
The first two properties are classical (see, for example, \cite[Theorem 3.28]{AFP}).

As for the third property, we first observe that, on the real line, Besicovitch differentiation theorem \cite[Theorem~2.22]{AFP} holds also with one-sided closed intervals instead of closed balls. Indeed, the proof is based on a covering argument that in the one-dimensional case works also with one-sided intervals. As a consequence, if $\mu \perp \nu$ are two Radon measures supported on disjoint sets, it turns out that
$$\lim_{h\to 0^+}\frac{\nu([x,x+h])}{\mu([x,x+h])} =\lim_{h\to 0^+}\frac{\nu([x,x+h))}{\mu([x,x+h))} = \lim_{h\to 0^-}\frac{\nu([x+h,x])}{\mu([x+h,x])}= \lim_{h\to 0^-}\frac{\nu((x+h,x])}{\mu((x+h,x])}=0$$
for $\mu$-almost every $x\in \R$.

Now let $D^c _+ u$ and $D^c _- u$ denote respectively the positive and negative part in the Hahn decomposition of $D^c u$. If we take the right continuous representative for $u$, then for every $x\in \R\setminus J_u$ and $h>0$ it holds that
$$\frac{u(x+h)-u(x)}{h}= \frac{Du([x,x+h])}{\Leb^1([x,x+h])}=\frac{D^c _+ u([x,x+h])}{\Leb^1([x,x+h])}\left[ 1+\frac{(D^a u+D^j u - D^c _- u)([x,x+h])}{D^c _+ u([x,x+h])}\right],$$
and
$$\frac{|D^c u ([x,x+h])|}{|D^c u| ([x,x+h])}=\frac{|(D^c _+ u - D^c _- u)([x,x+h])|}{(D^c_+ u + D^c _- u)([x,x+h])}=\frac{|1-D^c_- u ([x,x+h])/ D^c _+ u([x,x+h])|}{1+D^c_- u ([x,x+h])/ D^c _+ u([x,x+h])}, $$
while for $h<0$ it holds that
$$\frac{u(x+h)-u(x)}{h}= \frac{Du((x+h,x])}{\Leb^1((x+h,x])}=\frac{D^c _+ u((x+h,x])}{\Leb^1((x+h,x])}\left[ 1+\frac{(D^a u+D^j u - D^c _- u)((x+h,x])}{D^c _+ u((x+h,x])}\right],$$


Therefore, Besicovitch differentiation Theorem implies that $D^c _+ u$ is supported on $C_+$. With a similar argument we obtain also that $D^c _- u$ is supported on $C_-$.
\end{proof}

In the proof of Theorem~\ref{thm:lim_SBV} we also need a density result for piecewise smooth functions in $\dot{SBV}(\R)$. The precise class of functions that we consider is the following.

\begin{defn}
We denote with $X(\R)$ the set of functions $u\in \dot{SBV}(\R)$ such that $J_u$ is finite, $Du$ is compactly supported and $u\in C^\infty((a,b))\cap \Lip((a,b))$ for every open interval $(a,b)$ such that $(a,b) \cap J_u=\emptyset$.
\end{defn}

The strong density of $X(\R)$ into $\dot{SBV}(\R)$ is provided by the following lemma, which is elementary in the one-dimensional case (a similar statement in the higher dimensional case is proved in \cite{2017-DFP-Lincei}).

\begin{lemma}\label{lemma:density}
Let $u\in \dot{SBV}(\R)$. Then there exists a sequence of functions $\{u_n\}\subseteq X(\R)$ such that $\|Du_n-Du\|_\M\to 0$ as $n\to +\infty$.
\end{lemma}

\begin{proof}
By \cite[Corollary~3.33]{AFP} we can write $u=u^a+u^j$, where $u^a\in \dot{W}^{1,1}(\R)$ and $u^j$ is a pure jump function, namely $Du^j=\sum_{i\in I} \alpha_i \delta_{x_i}$, where $I$ is at most countable, $\delta_{x_i}$ denotes a Dirac delta in the point $x_i$, and $\{\alpha_i\}$ is a summable sequence (or a finite set) of real numbers. Therefore, we have that $Du=(u^a)'\Leb^1+Du^j$.

We can approximate separately the function $(u^a )' \in L^1(\R)$ with a sequence $\{v_n\} \subseteq C^\infty _c (\R)$ and the measure $Du^j$ with measures $\mu_n=\sum_{i\in I_n} \alpha_i \delta_{x_i}$ that are finite sums of Dirac masses, so that $\|v_n\Leb^1 + \mu_n - (u^a)'\Leb^1 - Du^j \|_\M\to 0$.

Now it is enough to choose a sequence of functions $\{u_n\}$ such that $Du_n=v_n\Leb^1 + \mu_n$.
\end{proof}

\subsection{Proof of Theorem~\ref{thm:liminf} when $N=1$}

Let $J=J_u$ and $C=C_+\cup C_-$ be as in Lemma~\ref{lemma:BV_1d} and let us set
$$A:=\{x\in\R:u \mbox{ is differentiable at }x\mbox{ and } u'(x)\neq 0\}.$$

\paragraph{\textmd{\textit{Approximation of the jump part}}}

Let us fix a small positive real number $\ep>0$, and let $J_\ep=\{s_1,\dots,s_{k_\ep}\}$ be a finite subset of $J$ such that
\begin{equation}\label{J-J_ep}
|Du|(J\setminus J_\ep)\leq\ep.
\end{equation}

For every $i\in \{1,\dots,k_\ep\}$ let us fix a positive radius $r_i>0$ in such a way that the following properties hold.
\begin{itemize}
\item $[s_{i_1}-r_{i_1},s_{i_1}+r_{i_1}]\cap [s_{i_2}-r_{i_2},s_{i_2}+r_{i_2}]=\emptyset$ for every $1\leq i_1< i_2\leq k_\ep$.

\item For every $i\in \{1,\dots,k_\ep\}$ it holds that
$$|u(y)-u(x)|>(1-\ep)|Du|(\{s_i\}) \qquad \forall (x,y)\in (s_i-r_i,s_i)\times(s_i,s_i+r_i).$$

\item If we set
$$U_\ep ^j:=\bigcup_{i=1} ^{k_\ep} [s_i-r_i,s_i+r_i],$$
then it holds that
$$|D^a u|(U_\ep ^j)+|D^c u|(U_\ep ^j)\leq \ep.$$
\end{itemize}

Now we observe that for every $i\in \{1,\dots, k_\ep\}$ we have that
$$|u(y)-u(x)|> (1-\ep)|Du|(\{s_i\})>\lambda |y-x|^{1+\gamma}$$
for every $(x,y)\in (s_i-r_i,s_i)\times(s_i,s_i+r_i)$ such that
$$|y-x|<\left(\frac{(1-\ep)|Du|(\{s_i\})}{\lambda}\right)^{\frac{1}{1+\gamma}}=:\delta_{\lambda,\ep}(s_i).$$

As a consequence, as soon as $\lambda>\lambda_\ep ^j:=\max\{(1-\ep)|Du|(\{s_i\})/r_i^{1+\gamma}:i\in \{1,\dots,k_\ep\}\}$, so that $\delta_{\lambda,\ep}(s_i)<r_i$ for every $i\in \{1,\dots,k_\ep\}$, we have that
\begin{equation}\label{defn:J_lambda,ep}
\mathcal{J}_{\lambda,\ep}:=\bigcup_{i=1} ^{k_\ep} \left\{(x,y)\in (s_i-\delta_{\lambda,\ep}(s_i),s_i)\times (s_i,s_i+r_i): y<x+\delta_{\lambda,\ep}(s_i)\right\} \subseteq E_{\gamma,\lambda}'(u) \cap (U_\ep ^j \times U_\ep ^j).
\end{equation}

\paragraph{\textmd{\textit{Approximation of the Cantor part}}}

Let us set $C_\ep:=C\setminus U_{\ep} ^j$. We observe that
\begin{equation}\label{est:C_ep}
|D^c u|(C_\ep)\geq |D^c u|(C)-|D^c u|(U_\ep ^j)\geq \|D^c u\|_\M-\ep.
\end{equation}

Since $U_\ep ^j$ is a closed set, we can find a neighborhood $U_\ep ^c$ of $C_\ep$ such that $U_\ep ^j\cap U_\ep ^c =\emptyset$ and
\begin{equation}\label{defn:U_ep^c}
\Leb^1(U_\ep ^c)\leq \ep,\qquad\mbox{and}\qquad |D^a u|(U_\ep ^c)\leq \ep.
\end{equation}

For every $\lambda\geq 0$ and every $z\in C_\ep\cap C_+$ let us set
\begin{align*}
r_{\lambda,\ep}(z):=\sup\biggl\{r>0:\ &\frac{u(z+h)-u(z)}{h^{1+\gamma}}\geq 2^{1+\gamma}\lambda \quad\mbox{and} \quad  u(z-h)\leq u(z) \quad \forall h\in(0,r),\\
&|D^c u([z,z+r])|\geq (1-\ep)|D^c u|([z,z+r]),\quad [z-r,z+r]\subseteq U_\ep ^c \biggr\},
\end{align*}
and similarly for $z\in C_\ep\cap C_-$ let us set
\begin{align*}
r_{\lambda,\ep}(z):=\sup\biggl\{r>0:\ &\frac{u(z+h)-u(z)}{h^{1+\gamma}}\leq -2^{1+\gamma}\lambda \quad\mbox{and} \quad  u(z-h)\geq u(z) \quad \forall h\in(0,r),\\
&|D^c u([z,z+r])|\geq (1-\ep)|D^c u|([z,z+r]),\quad [z-r,z+r]\subseteq U_\ep ^c \biggr\}.
\end{align*}

We claim that the number $r_{\lambda,\ep}(z)$ that we have defined satisfies the following properties.
\begin{itemize}
\item $r_{\lambda,\ep}(z)\in (0,\ep/2]$ for every $x\in C_\ep$ and every $\lambda\geq 0$.
\item The function $\lambda\mapsto r_{\lambda,\ep}(z)$ is nonincreasing on $[0,+\infty)$ for every $z\in C_\ep$ and
$$\lim_{\lambda\to +\infty}r_{\lambda,\ep}(z)=0 \qquad \forall z\in C_\ep.$$
\item If $r_{\lambda,\ep}(z)<r_{0,\ep}(z)$ for some $z\in C_\ep$ and some $\lambda>0$, then it holds that
\begin{equation}\label{eq:r_lambda>Du}
2^{1+\gamma}\lambda r_{\lambda,\ep}(z)^{1+\gamma} \geq |Du([z,z+r_{\lambda,\ep}(z)])|.
\end{equation}
\end{itemize}

Indeed, we have that $r_{\lambda,\ep}(z)>0$ because all the properties in the definition are verified when $r>0$ is sufficiently small, thanks to the definition of $C_\pm$ and the fact that $U_\ep ^c$ is a neighborhood of $C_\ep$, while $r_{\lambda,\ep}(z)\leq\ep/2$ thanks to (\ref{defn:U_ep^c}).

The monotonicity of $r_{\lambda,\ep}(z)$ with respect to $\lambda$ and the limit as $\lambda\to +\infty$ are immediate consequences of the first condition in the definition of $r_{\lambda,\ep}(z)$, which is also the only one involving $\lambda$.

The third property follows from the supremum property of $r_{\lambda,\ep}(z)$, and the fact that $|Du([z,z+r_{\lambda,\ep}(z)])|=|u((z+r_{\lambda,\ep}(z))_+)-u(z)|$. In fact all the properties in the definition of $r_{\lambda,\ep}(z)$ but the first one hold up to $r_{0,\ep}(z)$, because they are independent of $\lambda$, so if $r_{\lambda,\ep}(z)<r_{0,\ep}(z)$ then it means that the first property fails for a sequence of radii $r_n\searrow r_{\lambda,\ep}(z)$, and this implies (\ref{eq:r_lambda>Du}).

Now let us set $C_{\lambda,\ep}:=\{z\in C_\ep:r_{\lambda,\ep}(z)<r_{0,\ep}(z)\}$, so that (\ref{eq:r_lambda>Du}) holds for every $z\in C_{\lambda,\ep}$.

From the second property of $r_{\lambda,\ep}(z)$ we deduce that $C_{\lambda_1,\ep}\subseteq C_{\lambda_2,\ep}$ for every $0< \lambda_1<\lambda_2$ and that
$$\bigcup_{\lambda> 0} C_{\lambda,\ep} = C_\ep.$$

By \cite[Lemma~2.1]{1991-Aldaz}, for every $\lambda>0$  we can find a finite set $\{z_1,\dots,z_{m_{\lambda,\ep}}\}\subseteq C_{\lambda,\ep}$ such that
$$[z_{i_1},z_{i_1}+r_{\lambda,\ep}(z_{i_1})]\cap [z_{i_2},z_{i_2}+r_{\lambda,\ep}(z_{i_2})]=\emptyset \qquad \forall 1\leq i_1< i_2\leq m_{\lambda,\ep},$$
and
\begin{equation}\label{est:covering}
|D^c u|\Biggl(\bigcup_{i=1}^{m_{\lambda,\ep}} [z_i,z_i+r_{\lambda,\ep}(z_i)]\Biggr)\geq \frac{1}{2+\ep}|D^c u|\Biggl(\bigcup_{z\in C_{\lambda,\ep}} [z,z+r_{\lambda,\ep}(z)]\Biggr)\geq \frac{1}{2+\ep}|D^c u|(C_{\lambda,\ep}).
\end{equation}

Now we observe that if $z\in C_{\lambda,\ep}$, then for every $(x,y)\in (z-r_{\lambda,\ep}(z),z)\times (z,z+r_{\lambda,\ep}(z))$ we have that
$$|u(y)-u(x)|\geq|u(y)-u(z)|\geq 2^{1+\gamma}\lambda (y-z)^{1+\gamma} = \lambda (y-x)^{1+\gamma} \left(\frac{2(y-z)}{y-x}\right)^{1+\gamma}.$$

Therefore, for $(x,y)\in (z-r_{\lambda,\ep}(z),z)\times (z,z+r_{\lambda,\ep}(z))$, it holds that
$$y>2z-x\Longrightarrow \frac{2(y-z)}{y-x}>1 \Longrightarrow (x,y)\in E_{\gamma,\lambda}'(u).$$

This means that
\begin{align}
\mathcal{C}_{\lambda,\ep}&:=\bigcup_{i=1} ^{m_{\lambda,\ep}} \left\{(x,y)\in (z_i-r_{\lambda,\ep}(z_i),z_i)\times (z_i,z_i+r_{\lambda,\ep}(z_i)):y>2z_i -x \right\}\nonumber\\
&\subseteq E_{\gamma,\lambda}'(u) \cap (U_\ep ^c \times U_\ep ^c),\label{defn:C_lambda,ep}
\end{align}
for every $\lambda>0$.

\paragraph{\textmd{\textit{Approximation of the absolutely continuous part}}}

Let us set $A_\ep:=A\setminus (U_\ep ^j \cup U_\ep ^c)$. We observe that
\begin{equation}\label{est:A_ep}
|D^a u|(A_\ep)\geq \|D^a u\|_\M- |D^a u|(U_\ep ^j)-|D^a u|(U_\ep ^c)\geq \|D^a u\|_\M-2\ep.
\end{equation}

For every $\lambda>0$ let us set
$$A_{\lambda,\ep}:=\left\{x\in A_\ep : \frac{|u(y)-u(x)|}{y-x}>(1-\ep)|u'(x)| \ \ \forall y \in \left(x,x+\left(\frac{(1-\ep)|u'(x)|}{\lambda}\right)^{\frac{1}{\gamma}}\right)\right\}.$$

We observe that $A_{\lambda_1,\ep}\subseteq A_{\lambda_2,\ep}$ for every $0<\lambda_1<\lambda_2$ and that
$$\bigcup_{\lambda>0}A_{\lambda,\ep}=A_\ep.$$

Moreover, if $x\in A_{\lambda,\ep}$, then we have that
$$|u(y)-u(x)|>(1-\ep)(y-x)|u'(x)|>\lambda (y-x)^{1+\gamma},$$
for every
$$y \in \left(x,x+\left(\frac{(1-\ep)|u'(x)|}{\lambda}\right)^{\frac{1}{\gamma}}\right).$$

As a consequence, it turns out that
\begin{equation}\label{defn:A_lambda,ep}
\mathcal{A}_{\lambda,\ep}:=\left\{(x,y)\in A_{\lambda,\ep}\times \R : y \in \left(x,x+\left(\frac{(1-\ep)|u'(x)|}{\lambda}\right)^{\frac{1}{\gamma}}\right) \right\}
\subseteq E_{\gamma,\lambda}'(u) \setminus (U_\ep ^j \cup U_\ep ^c)\times \R,
\end{equation}
for every $\lambda>0$.

\paragraph{\textmd{\textit{Computation of the functional in the three parts}}}

From (\ref{F_E'}), (\ref{defn:J_lambda,ep}), (\ref{defn:C_lambda,ep}) and (\ref{defn:A_lambda,ep}) we deduce that
\begin{equation}\label{F=J+C+A}
F_{\gamma,\lambda}(u)\geq C_1\lambda \left(\nu_\gamma(\mathcal{J}_{\lambda,\ep})+ \nu_\gamma(\mathcal{C}_{\lambda,\ep})+ \nu_\gamma(\mathcal{A}_{\lambda,\ep}) \strut\right),
\end{equation}
for every $\lambda>\lambda_\ep ^j$.

The three addenda in the right-hand side can be computed as follows.
\begin{eqnarray*}
\nu_\gamma(\mathcal{J}_{\lambda,\ep})&=&
\sum_{i=1} ^{k_\ep} \int_{s_i - \delta_{\lambda,\ep}(s_i)} ^{s_i} dx \int_{s_i} ^{x+\delta_{\lambda,\ep}(s_i)} (y-x)^{\gamma-1}\,dy\\
&=& \sum_{i=1} ^{k_\ep} \frac{\delta_{\lambda,\ep}(s_i)^{1+\gamma}}{1+\gamma}\\
&=& \sum_{i=1} ^{k_\ep} \frac{(1-\ep)|Du|(\{s_i\})}{\lambda(1+\gamma)}\\
&=& \frac{1-\ep}{(1+\gamma)\lambda} |Du|(J_\ep).
\end{eqnarray*}

Recalling (\ref{J-J_ep}), we obtain that
\begin{equation}\label{est:J}
\nu_\gamma(\mathcal{J}_{\lambda,\ep})\geq \frac{1-\ep}{(1+\gamma)\lambda} (\|D^j u\|_\M-\ep).
\end{equation}

Now we compute the contribution of the Cantor part.
\begin{eqnarray*}
\nu_\gamma(\mathcal{C}_{\lambda,\ep})&=&
\sum_{i=1} ^{m_{\lambda,\ep}} \int_{z_i-r_{\lambda,\ep}(z_i)} ^{z_i} dx \int_{2z_i -x} ^{z_i+r_{\lambda,\ep}(z_i)} (y-x)^{\gamma-1}\,dy\\
&=& \sum_{i=1} ^{m_{\lambda,\ep}} \int_{z_i-r_{\lambda,\ep}(z_i)} ^{z_i} \frac{1}{\gamma}\left[(z_i+r_{\lambda,\ep}(z_i)-x)^\gamma -2^\gamma (z_i-x)^\gamma\right]dx\\
&=& \sum_{i=1} ^{m_{\lambda,\ep}} \frac{2^\gamma-1}{\gamma(1+\gamma)}r_{\lambda,\ep}(z_i)^{1+\gamma}.
\end{eqnarray*}

Recalling (\ref{eq:r_lambda>Du}), (\ref{J-J_ep}) and (\ref{defn:U_ep^c}) we obtain that
\begin{align*}
\nu_\gamma(\mathcal{C}_{\lambda,\ep})&\geq
\sum_{i=1} ^{m_{\lambda,\ep}}  \frac{2^\gamma-1}{\gamma(1+\gamma)2^{1+\gamma}} \frac{|Du([z_i,z_i+r_{\lambda,\ep}(z_i)])|}{\lambda}\\
&\geq \sum_{i=1} ^{m_{\lambda,\ep}}  \frac{2^\gamma-1}{\gamma(1+\gamma)2^{1+\gamma}} \frac{|D^c u([z_i,z_i+r_{\lambda,\ep}(z_i)])| -|(D^j u +D^a u) ([z_i,z_i+r_{\lambda,\ep}(z_i)])|}{\lambda}\\
&\geq \left[\sum_{i=1} ^{m_{\lambda,\ep}}  \frac{2^\gamma-1}{\gamma(1+\gamma)2^{1+\gamma}\lambda} |D^c u([z_i,z_i+r_{\lambda,\ep}(z_i)])| \right]
-\frac{|D^j u|(\R\setminus J_\ep)-|D^a u|(U_\ep ^c)}{\lambda} \\
&\geq \left[\sum_{i=1} ^{m_{\lambda,\ep}}  \frac{2^\gamma-1}{\gamma(1+\gamma)2^{1+\gamma}\lambda} |D^c u([z_i,z_i+r_{\lambda,\ep}(z_i)])| \right] - \frac{2\ep}{\lambda}.
\end{align*}

From the definition of $r_{\lambda,\ep}(z)$ we deduce that 
$$|D^c u([z_i,z_i+r_{\lambda,\ep}(z_i)])| \geq (1-\ep)|D^c u|([z_i,z_i+r_{\lambda,\ep}(z_i)]),$$
for every $i\in\{1,\dots,m_{\lambda,\ep}\}$.

As a consequence, from (\ref{est:covering}) we deduce that
\begin{eqnarray}
\nu_\gamma(\mathcal{C}_{\lambda,\ep})&\geq& \frac{(2^\gamma-1)(1-\ep)}{\gamma(1+\gamma)2^{1+\gamma}\lambda} |D^c u|\!\left(\bigcup_{i=1} ^{m_{\lambda,\ep}}[z_i,z_i+r_{\lambda,\ep}(z_i)]\right) -\frac{2\ep}{\lambda}\nonumber\\
&\geq& \frac{(2^\gamma-1)(1-\ep)}{\gamma(1+\gamma)2^{1+\gamma} (2+\ep)\lambda}|D^c u|(C_{\lambda,\ep})-\frac{2\ep}{\lambda}. \label{est:C}
\end{eqnarray}

Finally, we compute the contribution of the absolutely continuous part
\begin{eqnarray}
\nu_\gamma(\mathcal{A}_{\lambda,\ep}) &=&
\int_{A_{\lambda,\ep}} dx \int_{x} ^{x+\left(\frac{(1-\ep)|u'(x)|}{\lambda}\right)^{\frac{1}{\gamma}}} (y-x)^{\gamma-1} \,dy\nonumber\\
&=& \frac{1-\ep}{\gamma \lambda} \int_{A_{\lambda,\ep}} |u'(x)|\, dx \nonumber\\
&=& \frac{1-\ep}{\gamma \lambda} |D^a u|(A_{\lambda,\ep}). \label{est:A}
\end{eqnarray}

Plugging (\ref{est:J}), (\ref{est:C}) and (\ref{est:A}) into (\ref{F=J+C+A}) we obtain that
\begin{multline*}
F_{\gamma,\lambda}(u)\geq \\
C_1 \left(\frac{1-\ep}{1+\gamma}(\|D^j u\|_\M-\ep) + \frac{(2^\gamma-1)(1-\ep)}{\gamma(1+\gamma)2^{1+\gamma}(2+\ep)} |D^c u|(C_{\lambda,\ep}) -2\ep + \frac{1-\ep}{\gamma} |D^a u|(A_{\lambda,\ep}) \right),
\end{multline*}
and hence
\begin{multline*}
\liminf_{\lambda\to +\infty} F_{\gamma,\lambda}(u)\geq\\
 C_1 \left(\frac{1-\ep}{1+\gamma}(\|D^j u\|_\M-\ep) + \frac{(2^\gamma-1)(1-\ep)}{\gamma(1+\gamma)2^{1+\gamma}(2+\ep)} |D^c u|(C_{\ep}) -2\ep + \frac{1-\ep}{\gamma} |D^a u|(A_{\ep}) \right).
\end{multline*}

Finally, recalling (\ref{est:C_ep}) and (\ref{est:A_ep}) we deduce that
\begin{multline*}
\liminf_{\lambda\to +\infty} F_{\gamma,\lambda}(u)\geq\\
 C_1 \left(\frac{1-\ep}{1+\gamma}(\|D^j u\|_\M-\ep) + \frac{(2^\gamma-1)(1-\ep)}{\gamma(1+\gamma)2^{1+\gamma}(2+\ep)} \|D^c u\|_\M -3\ep+ \frac{1-\ep}{\gamma} (\|D^a u\|_\M-2\ep) \right).
\end{multline*}

Letting $\ep\to 0^+$ we obtain (\ref{th:liminf}) when $N=1$.
\qed

\subsection{Proof of Theorem~\ref{thm:lim_SBV} when $N=1$}

First of all, we observe that it is enough to show that
\begin{equation}\label{est:limsup}
\limsup_{\lambda\to +\infty} F_{\gamma,\lambda}(u) \leq \frac{C_1}{\gamma} \|D^{a}u\|_\M +\frac{C_1}{\gamma+1}\|D^{j}u\|_\M,
\end{equation}
because the opposite inequality is provided by Theorem~\ref{thm:liminf}, that we have already proved in the case $N=1$.

We divide the proof of (\ref{est:limsup}) into two steps. First we prove the result for functions $u\in X(\R)$, then we exploit Lemma~\ref{lemma:density} to extend the result to all $u\in \dot{SBV}(\R)$.

\paragraph{\textmd{\textit{Step 1: $u\in X(\R)$}}}

Let $s_1<s_2<\dots<s_n$ be the jump points of $u$, namely the finitely many elements of $J_u$, and let $(a,b)$ be an interval such that $Du$ is supported on $[a+1,b-1]$. Let us set $s_0:=a$ and $s_{n+1}:=b$.

With these definitions, we also set $\Omega_i:=(s_i,s_{i+1})$ for every $i\in \{0,\dots,n\}$,
$$d_0:=\min\left\{|s_{i+1}-s_i|:i\in\{0,\dots,n\}\strut\right\},\quad\mbox{and}\quad L:=\max\left\{\|u'\|_{L^\infty(\Omega_i)}:i\in \{0,\dots,n\}\right\}.$$

Now we observe that $|u(y)-u(x)|\leq \|Du\|_\M$ for every $(x,y)\in \R^2$, and therefore
\begin{equation}\label{est:y-x}
E_{\gamma,\lambda}'(u)\subseteq \left\{(x,y)\in \R^2:0<y-x<\left(\frac{\|Du\|_\M}{\lambda}\right)^{\frac{1}{1+\gamma}} \right\}.
\end{equation}

Let us set
$$\lambda_0:=\frac{\|Du\|_\M}{\min\{d_0^{1+\gamma},1\}},$$
so that for every $\lambda>\lambda_0$ we have that
\begin{equation}\label{est:lambda>lambda_0}
\left(\frac{\|Du\|_\M}{\lambda}\right)^{\frac{1}{1+\gamma}} <\min\{d_0,1\}.
\end{equation}

As a consequence, setting
$$R_\lambda ^i:=\left(s_i-\left(\frac{\|Du\|_\M}{\lambda}\right)^{\frac{1}{1+\gamma}},s_i \right)\times \left(s_i,s_i+\left(\frac{\|Du\|_\M}{\lambda}\right)^{\frac{1}{1+\gamma}} \right),$$
for every $\lambda>\lambda_0$ we have that
\begin{equation}\label{eq:E<Omega+R}
E_{\gamma,\lambda}'(u)\subseteq \bigcup_{i=0} ^{n} \Omega_i ^2 \cup \bigcup_{i=1} ^{n} R_{\lambda} ^i.
\end{equation}

We point out that (\ref{est:y-x}) and (\ref{est:lambda>lambda_0}) imply that points outside $[a,b]^2$ cannot contribute to $E_{\gamma,\lambda}(u)$ when $\lambda>\lambda_0$, because $u$ is constant in $(-\infty,a+1)$ and in $(b-1,+\infty)$.

Now let us fix $\ep>0$ and let us set
$$\Omega_{\lambda,\ep} ^i:=\left\{x\in\Omega_i: \frac{|u(y)-u(x)|}{y-x}\leq |u'(x)|+\ep \quad \forall y \in \left(x,x+\left(\frac{\|Du\|_\M}{\lambda}\right)^{\frac{1}{1+\gamma}}\right)\cap\Omega_i \right\}.$$

We observe that $\Omega_{\lambda_1,\ep} ^i \subseteq \Omega_{\lambda_2,\ep} ^i$ for every $0<\lambda_1<\lambda_2$, and that
$$\bigcup_{\lambda>0} \Omega_{\lambda,\ep} ^i \supseteq \Omega_i.$$

Moreover, if $(x,y)\in E_{\gamma,\lambda}'(u)\cap \Omega_i ^2$ for some $x\in\Omega_{\lambda,\ep}^i$, then we claim that
$$y<x+\left(\frac{|u'(x)|+\ep}{\lambda}\right)^{\frac{1}{\gamma}}.$$

Indeed, if this is not the case, then from (\ref{est:y-x}) and the definition of $\Omega_{\lambda,\ep}^i$ we deduce that
$$|u(y)-u(x)|\leq (y-x)(|u'(x)|+\ep)\leq \lambda (y-x)^{1+\gamma},$$
which contradicts the fact that $(x,y)\in E_{\gamma,\lambda}'(u)$.

On the other hand, since $u$ is Lipschitz continuous with Lipschitz constant bounded by $L$ in each of the intervals $\Omega_i$, for every $(x,y)\in E_{\gamma,\lambda}'(u)\cap \Omega_i ^2$ it holds that
$$y<x+\left(\frac{L}{\lambda}\right)^{\frac{1}{\gamma}},$$
otherwise we would have that $|u(y)-u(x)|\leq (y-x)L\leq \lambda (y-x)^{1+\gamma}$.

Hence we obtain the following estimate
\begin{eqnarray}
\nu_\gamma(E_{\gamma,\lambda}'(u)\cap \Omega_i ^2)&\leq&
\int_{\Omega_{\lambda,\ep}^i}dx \int_{x} ^{x+\left(\frac{|u'(x)|+\ep}{\lambda}\right)^{\frac{1}{\gamma}}} (y-x)^{\gamma-1}\,dy\nonumber\\
&&\mbox{}+\int_{\Omega_i\setminus\Omega_{\lambda,\ep}^i}dx \int_{x} ^{x+\left(\frac{L}{\lambda}\right)^{\frac{1}{\gamma}}} (y-x)^{\gamma-1}\,dy\nonumber\\
&=& \frac{1}{\gamma\lambda}\left[\int_{\Omega_{\lambda,\ep}^i}|u'(x)|\,dx +\ep\Leb^1(\Omega_{\lambda,\ep}^i)+ L\Leb^1(\Omega_i\setminus\Omega_{\lambda,\ep}^i)\right].\label{est:Omega_i}
\end{eqnarray}

Now we need to estimate the contribution of the rectangles $R_\lambda ^i$. To this end, let $\lambda^j_\ep>0$ be a positive number large enough so that for every $i\in \{1,\dots,n\}$ and every $\lambda>\lambda_\ep ^j$ it holds that
$$|u(y)-u(x)|\leq (1+\ep)|Du|(\{s_i\}) \qquad \forall (x,y)\in R_{\lambda} ^i.$$

Then we have that
$$|u(y)-u(x)|\leq (1+\ep)|Du|(\{s_i\}) \leq \lambda (y-x)^{1+\gamma},$$
for every $(x,y)\in R_\lambda ^i$ such that
$$y\geq x+\left(\frac{(1+\ep)|Du|(\{s_i\})}{\lambda}\right)^{\frac{1}{1+\gamma}}.$$

Therefore we deduce that
\begin{eqnarray*}
E_{\gamma,\lambda}'(u)\cap R_\lambda ^i \subseteq \left\{(x,y)\in \R^2 :
x\in \left(s_i-\left(\frac{(1+\ep)|Du|(\{s_i\})}{\lambda}\right)^{\frac{1}{1+\gamma}},s_i\right),\right.\\
\left. y\in \left(s_i,x+\left(\frac{(1+\ep)|Du|(\{s_i\})}{\lambda}\right)^{\frac{1}{1+\gamma}}\right)\right\}.
\end{eqnarray*}

Hence we obtain the following estimate
\begin{eqnarray}
\nu_\gamma(E_{\gamma,\lambda}'(u)\cap R_\lambda ^i) &\leq& \int_{s_i-\left(\frac{(1+\ep)|Du|(\{s_i\})}{\lambda}\right)^{\frac{1}{1+\gamma}}} ^{s_i} dx \int_{s_i} ^{x+\left(\frac{(1+\ep)|Du|(\{s_i\})}{\lambda}\right)^{\frac{1}{1+\gamma}}} (y-x)^{\gamma-1}\,dy\nonumber\\
&=&\frac{(1+\ep)|Du|(\{s_i\})}{(1+\gamma)\lambda}.\label{est:R_i}
\end{eqnarray}

Finally, combining (\ref{F_E'}), (\ref{eq:E<Omega+R}), (\ref{est:Omega_i}) and (\ref{est:R_i}), for every $\lambda>\max\{\lambda_0,\lambda_\ep ^j\}$ we obtain that
$$F_{\gamma,\lambda}(u)\leq \sum_{i=0} ^n \frac{C_1}{\gamma}\left[\int_{\Omega_{\lambda,\ep}^i}|u'(x)|\,dx +\ep\Leb^1(\Omega_{\lambda,\ep}^i)+ L\Leb^1(\Omega_i\setminus\Omega_{\lambda,\ep}^i)\right] + \sum_{i=1} ^n \frac{C_1 (1+\ep)|Du|(\{s_i\})}{(1+\gamma)},$$
and hence
$$\limsup_{\lambda\to +\infty} F_{\gamma,\lambda}(u)\leq \frac{C_1}{\gamma}\left[\|D^a u\|_\M +\ep(b-a)\right] + \frac{C_1 (1+\ep)\|D^j u\|_\M}{(1+\gamma)}.$$

Letting $\ep\to 0^+$, we obtain (\ref{est:limsup}).

\paragraph{\textmd{\textit{Step 2: $u\in \dot{SBV}(\R)$}}}

We exploit the very same argument used in \cite[Section~3.4]{BSVY} to reduce the case $u\in \dot{W}^{1,1}(\R^N)$ to the case in which $u$ is smooth and has compactly supported gradient.

So, given $u\in \dot{SBV}(\R)$, let $\{u_n\}$ be a sequence provided by Lemma~\ref{lemma:density}. By the triangle inequality for every $n\in\N$ and every $\ep\in (0,1)$ it holds that
$$E_{\gamma,\lambda}'(u)\subseteq E_{\gamma,\lambda}'(u_n/(1-\ep)) \cup E_{\gamma,\lambda}'((u-u_n)/\ep).$$

Recalling (\ref{F_E'}), it follows that
$$F_{\gamma,\lambda}(u)\leq F_{\gamma,\lambda}(u_n/(1-\ep)) + F_{\gamma,\lambda}((u-u_n)/\ep),$$
and hence
$$\limsup_{\lambda\to +\infty} F_{\gamma,\lambda}(u)\leq \limsup_{\lambda\to +\infty} F_{\gamma,\lambda}(u_n/(1-\ep)) + \sup_{\lambda>0} F_{\gamma,\lambda}((u-u_n)/\ep).$$

Since $u_n\in X(\R)$, by Step~1 and the second inequality in (\ref{est:sup_F}) we conclude that
$$\limsup_{\lambda\to +\infty} F_{\gamma,\lambda}(u)\leq \left[\frac{C_1}{\gamma} \frac{\|D^{a}u_n\|_\M}{1-\ep}+\frac{C_1}{\gamma+1}\frac{\|D^{j}u_n\|_\M}{1-\ep}\right] + c_2(N,\gamma)\frac{\|D(u-u_n)\|_\M}{\ep}.$$

Letting first $n\to +\infty$, and then $\ep \to 0^+$ we obtain (\ref{est:limsup}).
\qed

\section{The higher dimensional case}\label{sec:Nd}

In this section we extend the results of Section~\ref{sec:1d} to the case $N>1$, thus establishing Theorem~\ref{thm:liminf} and Theorem~\ref{thm:lim_SBV} in full generality.

The main tool that we exploit is a representation formula for $F_{\gamma,\lambda}(u)$, which allows us to rewrite this functional in terms of its one-dimensional version computed on all one-dimensional sections of the function $u$.

In order to state the formula, let us introduce some notation. For a function $u:\R^N\to\R$, a unit vector $\sigma \in \s^{N-1}$ and a point $z\in \sigma^\perp$ let $u_{\sigma,z}:\R\to\R$ be the one-dimensional function coinciding with the restriction of $u$ to the line parallel to $\sigma$ passing through $z$, namely the function
$$u_{\sigma,z}(t):=u(z+\sigma t) \qquad \forall t \in \R.$$

The result is the following.

\begin{lemma}\label{lemma:integral_geometric_repr}
Let $u:\R^N\to\R$ be a measurable function and let us fix $\gamma>0$ and $\lambda>0$.

Then it turns out that
$$F_{\gamma,\lambda}(u)=\frac{1}{C_1}\int_{\s^{N-1}}d\sigma \int_{\sigma^\perp} F_{\gamma,\lambda}(u_{\sigma,z})\,dz.$$
\end{lemma}

\begin{proof}
Let $\chi:\R^N\times \R^N\to \{0,1\}$ be the characteristic function of the set $E_{\gamma,\lambda}(u)$. Then we have that
$$F_{\gamma,\lambda}(u)=\int_{\R^N}dx\int_{\R^N}|y-x|^{\gamma-N}\chi(x,y)\, dy.$$

Setting first $y=x+\sigma r$, with $\sigma \in \s^{N-1}$ and $r\in (0,+\infty)$, and then $x=z+\sigma t$, with $z\in \sigma^\perp$ and $t\in\R$, we obtain that
\begin{eqnarray*}
F_{\gamma,\lambda}(u)&=&\int_{\R^N}dx\int_{\s^{N-1}}d\sigma \int_{0} ^{+\infty} r^{\gamma-N}\chi(x,x+\sigma r) r^{N-1}\, dr\\
&=&\int_{\s^{N-1}}d\sigma \int_{\sigma^\perp} dz \int_{\R} dt \int_{0} ^{+\infty} r^{\gamma-1}\chi(z+\sigma t,z+\sigma (t+r))\,dr\\
&=&\frac{1}{2}\int_{\s^{N-1}}d\sigma \int_{\sigma^\perp} dz \int_{\R} dt \int_{\R} |r|^{\gamma-1}\chi(z+\sigma t,z+\sigma (t+r))\,dr.
\end{eqnarray*}

Finally, setting $r=s-t$ and recalling that $C_1=2$ we obtain that
\begin{eqnarray*}
F_{\gamma,\lambda}(u)&=&\frac{1}{2}\int_{\s^{N-1}}d\sigma \int_{\sigma^\perp} dz \int_{\R} dt \int_{\R} |s-t|^{\gamma-1}\chi(z+\sigma t,z+\sigma s)\,ds\\
&=&\frac{1}{C_1}\int_{\s^{N-1}}d\sigma \int_{\sigma^\perp} F_{\gamma,\lambda}(u_{\sigma,z})\,dz.
\end{eqnarray*}
\end{proof}

The second result that we need to perform the sectioning argument is the following well-known lemma, that follows from the results in \cite[Section~3.11]{AFP}.

\begin{lemma}\label{lemma:sectioning}
Let $u\in \dot{BV}(\R^N)$. Then for every $\sigma \in \s^{N-1}$ it holds that $u_{\sigma,z}\in \dot{BV}(\R)$ for almost every $z \in \sigma^\perp$ and
\begin{gather*}
\int_{\s^{N-1}}d\sigma \int_{\sigma^\perp} \|D^a u_{\sigma,z}\|_\M\,dz=C_N\|D^a u\|_\M,\\[0.5ex]
\int_{\s^{N-1}}d\sigma \int_{\sigma^\perp} \|D^j u_{\sigma,z}\|_\M\,dz=C_N\|D^j u\|_\M,\\[0.5ex]
\int_{\s^{N-1}}d\sigma \int_{\sigma^\perp} \|D^c u_{\sigma,z}\|_\M\,dz=C_N\|D^c u\|_\M.
\end{gather*}
\end{lemma}

We can now extend the proofs of our main results to the case $N>1$.

\begin{proof}[Proof of Theorem~\ref{thm:liminf}]
It is enough to apply consecutively Lemma~\ref{lemma:integral_geometric_repr}, Fatou lemma, the one-dimensional result and Lemma~\ref{lemma:sectioning}, to obtain that
\begin{eqnarray*}
\liminf_{\lambda\to +\infty} F_{\gamma,\lambda}(u)&=&
\liminf_{\lambda\to +\infty} \frac{1}{C_1}\int_{\s^{N-1}}d\sigma \int_{\sigma^\perp} F_{\gamma,\lambda}(u_{\sigma,z})\,dz\\
&\geq& \int_{\s^{N-1}}d\sigma \int_{\sigma^\perp}\frac{1}{C_1} \liminf_{\lambda\to +\infty} F_{\gamma,\lambda}(u_{\sigma,z})\,dz\\
&\geq& \int_{\s^{N-1}} d\sigma \int_{\sigma^\perp} \left(\frac{\|D^a u_{\sigma,z}\|_\M}{\gamma} + \frac{\|D^j u_{\sigma,z}\|_\M}{1+\gamma} + \frac{(2^\gamma-1)\|D^c u_{\sigma,z}\|_\M}{\gamma(1+\gamma)2^{2+\gamma}} \right)dz\\
&=& \frac{C_N}{\gamma} \|D^{a}u\|_\M+\frac{C_N}{1+\gamma}\|D^{j}u\|_\M+ \frac{C_N(2^\gamma-1)}{\gamma(1+\gamma)2^{2+\gamma}}\|D^{c}u\|_\M.
\end{eqnarray*}
\end{proof}

\begin{proof}[Proof of Theorem~\ref{thm:lim_SBV}]
From (\ref{est:sup_F}) and Lemma~\ref{lemma:sectioning} we deduce that
$$\int_{\s^{N-1}}d\sigma \int_{\sigma^\perp} \sup_{\lambda>0} F_{\gamma,\lambda} (u_{\sigma,z})\,dz \leq
\int_{\s^{N-1}}d\sigma \int_{\sigma^\perp} c_2 \|D u_{\sigma,z}\|_\M\,dz= c_2 C_N \|D u\|_\M<+\infty.
$$

Therefore, from Lemma~\ref{lemma:integral_geometric_repr}, the dominated convergence theorem, the one-dimensional result and Lemma~\ref{lemma:sectioning}, we deduce that
\begin{eqnarray*}
\lim_{\lambda\to +\infty} F_{\gamma,\lambda}(u)&=&
\lim_{\lambda\to +\infty} \frac{1}{C_1}\int_{\s^{N-1}}d\sigma \int_{\sigma^\perp} F_{\gamma,\lambda}(u_{\sigma,z})\,dz\\
&=& \int_{\s^{N-1}}d\sigma \int_{\sigma^\perp}\frac{1}{C_1} \lim_{\lambda\to +\infty} F_{\gamma,\lambda}(u_{\sigma,z})\,dz\\
&=& \int_{\s^{N-1}} d\sigma \int_{\sigma^\perp} \left(\frac{\|D^a u_{\sigma,z}\|_\M}{\gamma} + \frac{\|D^j u_{\sigma,z}\|_\M}{1+\gamma} \right)dz\\
&=& \frac{C_N}{\gamma} \|D^{a}u\|_\M+\frac{C_N}{1+\gamma}\|D^{j}u\|_\M.
\end{eqnarray*}
\end{proof}

\subsubsection*{\centering Acknowledgments}

The author wishes to thank prof. H.~Brezis for sending him preliminary versions of \cite{Brezis-Lincei,BSVY} and for encouraging him to investigate the open problems stated therein.

The author is a member of the \selectlanguage{italian} ``Gruppo Nazionale per l'Analisi Matematica, la Probabilità e le loro Applicazioni'' (GNAMPA) of the ``Istituto Nazionale di Alta Matematica'' (INdAM).

\selectlanguage{english}

The author acknowledges the MIUR Excellence Department Project awarded to the Department of Mathematics, University of Pisa, CUP I57G22000700001.

\label{NumeroPagine}


\begin{thebibliography}{10}


\bibitem{1991-Aldaz}
\textsc{J.~M.~Aldaz}.
\newblock A general covering lemma for the real line.
\newblock \emph{Real Anal. Exchange} \textbf{17} (1991/92), no~1, 394--398.


\bibitem{AFP}
\textsc{L.~Ambrosio}, \textsc{N.~Fusco}, \textsc{D.~Pallara}.
\newblock {Functions of bounded variation and free discontinuity
  problems}.
\newblock \emph{Oxford Mathematical Monographs}, 2000.

\bibitem{2018-CRAS-AGMP}
\textsc{C.~Antonucci}, \textsc{M.~Gobbino}, \textsc{M.~Migliorini}, \textsc{N.~Picenni}.
\newblock On the shape factor of interaction laws for a non-local approximation of the Sobolev norm and the total variation.
\newblock \emph{C. R. Math. Acad. Sci. Paris} \textbf{356} (2018), no.~8, 859--864.

\bibitem{2020-APDE-AGMP}
\textsc{C.~Antonucci}, \textsc{M.~Gobbino}, \textsc{M.~Migliorini}, \textsc{N.~Picenni}.
\newblock Optimal constants for a nonlocal approximation of Sobolev norms and total variation.
\newblock \emph{Anal. PDE} \textbf{13} (2020), no.~2, 595--625.

\bibitem{2020-APDE-AGP}
\textsc{C.~Antonucci}, \textsc{M.~Gobbino}, \textsc{N.~Picenni}.
\newblock On the gap between the Gamma-limit and the pointwise limit for a nonlocal approximation of the total variation.
\newblock \emph{Anal. PDE} \textbf{13} (2020), no.~3, 627--649.

\bibitem{BBM}
\textsc{J.~Bourgain}, \textsc{H.~Brezis}, and \textsc{P.~Mironescu}.
\newblock Another look at Sobolev spaces.
\newblock \emph{Optimal control and partial
differential equations, IOS, Amsterdam} (2001), 439--455.

\bibitem{Brezis-Lincei}
\textsc{H.~Brezis}.
\newblock Some of my favorite open problems.
\newblock \emph{Atti Accad. Naz. Lincei Rend. Lincei Mat. Appl.} \textbf{34} (2023), no.~2, 307--335.

\bibitem{2018-AnnPDE-BN}
\textsc{H.~Brezis}, \textsc{H.-M.~Nguyen}.
\newblock Non-local functionals related to the total variation and connections with image processing.
\newblock \emph{Ann. PDE} \textbf{4} (2018), no.~1, 9, 77 pp.

\bibitem{2021-BVY-PNAS}
\textsc{H.~Brezis}, \textsc{J.~Van Schaftingen}, \textsc{P.-L.~Yung}.
\newblock A surprising formula for Sobolev norms.
\newblock \emph{Proc. Nat. Acad. Sci. USA} \textbf{118} (2021), no.~8, e2025254118, 6 pp.

\bibitem{2021-BVY-CalcVar}
\textsc{H.~Brezis}, \textsc{J.~Van Schaftingen}, \textsc{P.-L.~Yung}.
\newblock Going to Lorentz when fractional Sobolev, Gagliardo and Nirenberg estimates fail.
\newblock \emph{Calc. Var. Partial Differential
Equations} \textbf{60} (2021), no.~4, 129, 12 pp.

\bibitem{2022-BSVY-Lincei}
\textsc{H.~Brezis}, \textsc{A.~Seeger}, \textsc{J.~Van~Schaftingen}, \textsc{P.-L.~Yung}.
\newblock Sobolev spaces revisited.
\newblock \emph{Atti Accad. Naz. Lincei Rend. Lincei Mat. Appl.} \textbf{33} (2022), no.~2, 413--437.

\bibitem{BSVY}
\textsc{H.~Brezis}, \textsc{A.~Seeger}, \textsc{J.~Van~Schaftingen}, \textsc{P.-L.~Yung}.
\newblock Families of functionals representing Sobolev norms.
\newblock \emph{Anal. PDE}, to appear.

\bibitem{DLYYZ_JFA}
\textsc{F.~Dai}, \textsc{X.~Lin}, \textsc{D.~Yang}, \textsc{W.~Yuan}, \textsc{Y.~Zhang}.
\newblock Poincaré inequality meets Brezis–Van Schaftingen–Yung formula on metric measure spaces.
\newblock \emph{J. Funct. Anal.} \textbf{283} (2022), Paper No.~109645, 52 pp.

\bibitem{DLYYZ_CalcVar}
\textsc{F.~Dai}, \textsc{X.~Lin}, \textsc{D.~Yang}, \textsc{W.~Yuan}, \textsc{Y.~Zhang}.
\newblock Brezis–Van Schaftingen–Yung formulae in ball Banach function spaces with applications to fractional Sobolev and Gagliardo–Nirenberg inequalities.
\newblock{Calc. Var. Partial Differential Equations} \textbf{62} (2023), Paper No.~56, 73 pp.

\bibitem{Davila-BBM}
\textsc{J.~Dávila}.
\newblock On an open question about functions of bounded variation.
\newblock \emph{Calc. Var. Partial Differential
Equations} \textbf{15} (2002), no.~4, 519--527.

\bibitem{2017-DFP-Lincei}
\textsc{G.~De Philippis}, \textsc{N.~Fusco}, \textsc{A.~Pratelli}.
\newblock On the approximation of SBV functions.
\newblock \emph{Atti Accad. Naz. Lincei Rend. Lincei Mat. Appl.} \textbf{28} (2017), no.~2, 369--413.

\bibitem{2022-GP-CCM}
\textsc{M.~Gobbino}, \textsc{N.~Picenni}.
\newblock On the characterization of constant functions through nonlocal functionals.
\newblock \emph{Comm. Contemp. Math.} \textbf{25} (2023), no.~9, Paper No.~2250038, 30 pp.

\bibitem{2023-GP-Gamma-liminf}
\textsc{M.~Gobbino}, \textsc{N.~Picenni}.
\newblock Gamma-liminf estimate for a class of non-local approximations of Sobolev and BV norms.
\newblock arXiv:2311.05560 (2023).

\bibitem{2023-Lahti}
\textsc{P.~Lahti}.
\newblock A sharp lower bound for a class of non-local approximations of the total variation.
\newblock arXiv:2310.03550 (2023).

\bibitem{2006-Nguyen-JFA}
\textsc{H.-M.~Nguyen}.
\newblock Some new characterizations of Sobolev spaces.
\newblock \emph{J. Funct. Anal.} \textbf{237} (2006), no.~2, 689--720.

\bibitem{2011-Nguyen-Duke}
\textsc{H.-M.~Nguyen}.
\newblock $\Gamma$-convergence, Sobolev norms, and BV functions.
\newblock \emph{Duke Math. J.} \textbf{157} (2011), no.~3, 495--533.

\bibitem{2022-Poliakovsky-JFA}
\textsc{A.~Poliakovsky}.
\newblock Some remarks on a formula for Sobolev norms due to Brezis, Van Schaftingen and Yung.
\newblock \emph{J. Funct. Anal.} \textbf{282} (2022), no.~3, Paper No.~109312, 47 pp.

\bibitem{ZYY_CCM}
\textsc{C.~Zhu}, \textsc{D.~Yang}, \textsc{W.~Yuan}.
\newblock Brezis–Seeger–Van Schaftingen–Yung-type characterization of homogeneous ball Banach Sobolev spaces and its applications.
\newblock \emph{Commun. Contemp. Math.} (2023), DOI: https://doi.org/10.1142/S0219199723500414.

\bibitem{ZYY_CalcVar}
\textsc{C.~Zhu}, \textsc{D.~Yang}, \textsc{W.~Yuan}.
\newblock Generalized Brezis–Seeger–Van Schaftingen–Yung formulae and their applications in ball Banach Sobolev spaces.
\newblock \emph{Calc. Var. Partial Differential Equations}
\textbf{62} (2023), Paper No.~234, 76 pp.

\end{thebibliography}
\end{document}